\documentclass[12pt,twoside,reqno]{amsart}
\usepackage{amsmath, amsfonts, amsthm, amssymb,amscd, graphicx, amscd}
\usepackage{float,epsf}
\usepackage[english]{babel}
\usepackage{enumerate}
\usepackage{tikz}

\usepackage{srcltx}
\usepackage{geometry}
\usepackage{verbatim}
\usepackage{mathrsfs}
\usepackage{hyperref}
\usepackage{enumitem} 
\usepackage{esint}

\newcommand{\R}{\mathbb{R}}

\newcommand{\M}{\mathcal{M}}

\newcommand{\tn}{\textnormal}

\newcommand{\eps}{\varepsilon}

 \newcommand{\supp}{\text{\rm supp}\,}

 \renewcommand{\supp}{\text{\rm supp}\,}

\newtheorem{theorem}{Theorem}[section]
\newtheorem{lemma}[theorem]{Lemma}

\newtheorem{corollary}[theorem]{Corollary}

\numberwithin{equation}{section}
\numberwithin{figure}{section}

\newcommand{\norm}[1]{\left\|#1\right\|}

\newcommand{\field}[1]{\mathbb{#1}}

\newcommand{\N}{\field{N}}

\def\intave#1{\int_{#1}\hbox{\llap{$\raise2.3pt\hbox{\vrule
height.9pt width7pt}\phantom{\scriptstyle{#1}}\mkern-2mu$}}}

\begin{document}
\title[A nonlinear model for long-range segregation]{A nonlinear model for long-range segregation}

\author{Howen Chuah}
\address{Howen Chuah,
Department of Mathematics, Purdue University,
 150 N. University Street,
 West Lafayette, IN 47907-2067, USA}
\email{hchuah@purdue.edu}

\author{Stefania Patrizi}
\address{Stefania Patrizi,
Department of Mathematics, University of Texas at Austin,  2515 Speedway, Austin, TX 78712, USA}
\email{spatrizi@math.utexas.edu}

\author{Monica Torres}
\address{Monica Torres,
Department of Mathematics, Purdue University,
 150 N. University Street,
 West Lafayette, IN 47907-2067, USA}
 \email{torres@math.purdue.edu}


\maketitle

\begin{center}{\em Dedicated to Prof.  Sandro Salsa for his 75th birthday}\end{center}

\begin{abstract}
We study a system of fully nonlinear elliptic equations, depending on a small parameter  $\eps$,  that models long-range segregation of populations. 
The diffusion is governed by the negative Pucci operator. In the linear case, this system was previously investigated by Caffarelli, Patrizi, and Quitalo in \cite{CL2} 
as a model in population dynamics. 
We establish the existence of solutions and prove convergence  as  $\eps\to0^+$  to a free boundary problem  in which populations remain segregated at a positive distance. 
In addition, we show that the supports of the  limiting functions  are sets of finite perimeter and satisfy a semi-convexity property.
\,\\\end{abstract}

\vspace{-6mm}
\begingroup\hypersetup{hidelinks} 
\endgroup
\newcommand{\sizeofboundary}{R}

\section{Introduction}
In this paper, we study the following fully nonlinear system of elliptic equations: for  $i=1,\ldots, K$, 

\begin{equation}
\label{main_problems}
\begin{cases} \M^-(u^{\varepsilon}_{i})= \frac{1}{\eps^2} u^{\eps}_i  \sum_{j \neq i}  H_R(u^{\eps}_j)(x)\quad & \text{ in } \Omega,\\
        u^{\eps}_i= f_i & \text{ on } (\partial \Omega)_{\leq {\sizeofboundary}}, \\
        u^{\eps}_i \geq 0  & \text{ in } \Omega \cup (\partial \Omega)_{\leq \sizeofboundary},
       \end{cases}
\end{equation}
where $\Omega$ is a bounded Lipschitz domain in  $\R^n$, $\eps >0$, $0 < \sizeofboundary \leq 1$. The boundary neighborhood  is defined as 
\begin{equation*}
    (\partial \Omega)_{\leq \sizeofboundary}:= \{x \in \Omega^c: d(x, \partial \Omega)\leq \sizeofboundary\} ,
\end{equation*}
where $d(x,\partial \Omega) := \inf_{y \in \partial \Omega}|x-y|$ denotes the distance of  $x$ from  $\partial \Omega$.  The boundary data are nonnegative H{\"o}lder continuous  functions
 with supports separated by at least distance $R$ (see assumptions \eqref{fiassumption}-\eqref{f_i-f_j_disjointsupport}).  

The operator $\mathcal{M}^{-}$  is the negative extremal Pucci operator, which is uniformly elliptic and fully nonlinear 
(see Section \ref{prelsection} for its  definition). Each equation in the system is coupled to the others through a nonlocal zero-order interaction  term
$H_R(u_j^\eps)$, which depends  on the parameter $R$. 
We consider two different  cases  for $H_R$:  for $w\geq 0$,


\begin{equation}
\label{H1}
H_R(w)(x)= \fint_{{B}_\sizeofboundary(x)}w^p(y)\,dy,\quad p\geq 1, 
\end{equation}
and 
\begin{equation}
\label{H2}
 H_R(w)(x)= \sup_{{B}_R(x)} w. 
\end{equation}

 We prove the existence of a positive  solution $(u_1^\eps,\ldots,u_K^\eps)$ to the system   \eqref{main_problems},    and show that, up to a subsequence, these solutions converge as $\eps\to0^+$ to a limit 
configuration $(u_1,\ldots, u_K)$. The supports of the limit functions $u_i$  are mutually disjoint and are separated from each other by a distance of at least $R$. 
Furthermore, we begin the study of the geometric properties of the boundaries of these supports within $\Omega$, the so-called free boundaries. Many of the proofs presented here  are adaptations of arguments developed by Caffarelli, Patrizi, and Quitalo in \cite{CL2}, where system \eqref{main_problems} is studied in the case of the Laplace operator. 

A deeper analysis of the free boundary regularity, as well as the asymptotic behavior of solutions as $R\to0^+$, will be addressed in a forthcoming paper.

\subsection{Background and motivations}

In population dynamics, Gause-Lotka-Voltera  models describe coexistence of species that live in the same territory, diffuse, and compete for limited resources.

  One of the simplest forms of such models consists of a system of equations of the form: for $i=1,\ldots,K$ and $\eps>0$,
\begin{equation}\label{modelproblem}L_i(u^\eps _i)=\frac{u^\eps_i}{\eps^2}F(u^\eps_1,\ldots,u^\eps_K),\end{equation}
in some domain $\Omega$, where $u^\eps_i$ is a positive function representing the density of the $i$-th species, $L_i$ encodes  the diffusion of  $u^\eps_i$, and   $u_i^\eps F(u^\eps_1,\ldots,u^\eps_K)/\eps^2$ models  the  attrition of the species $i$ due to competition with the others.  The interaction functional  $F$ is strictly positive whenever the supports of two or more species overlap.
The smaller the parameter $\eps$, the stronger the competition among species. 
In the limit as $\eps\to0^+$ the high  competition  forces the species  to   segregate, meaning  $u_iu_j=0$ for $j\neq i$,  in a such a way that $u_i F(u_1,\ldots,u_K)=0. $

A classical example of system \eqref{modelproblem} is given by: for $i=1,\ldots,K$ and $\eps>0$,
\begin{equation}\label{adjacentsegregation}\Delta u^\eps _i=\frac{1}{\eps^2}\sum_{j\neq i}u^\eps_i u_j^\eps. \end{equation}

The existence of positive solutions to \eqref{adjacentsegregation} was initially investigated by Dancer and Du \cite{DancerDu1, DancerDu2} in the case of three species. 
Convergence to a segregated limit configuration as  $\eps\to0^+$ was  later proven by Dancer, Hilhorst, Mimura, and Peletier \cite{DancerHilMimPel}. 
 More general classes of linear competitive systems, including \eqref{adjacentsegregation} as a special case, have been studied by Conti, Terracini, and Verzini \cite{Conti4, Conti3, Conti5}. For related optimal partition problems involving the first eigenvalue of the Laplace operator, see also  \cite{CL7, Conti6}.  
 The geometric properties of the free boundaries $\partial\{u_i>0\}\cap\Omega$ have been investigated  by Caffarelli, Karakhanyan and   Lin  \cite{CL5} (see also \cite{CL4}).  There, it is shown that
  each free boundary splits into two parts: a regular set,  which is a locally  analytic surface,  and  a singular set,  which is a closed set of Hausdorff dimension at most $n-2$.
   Singular points occur where the boundaries of three or more connected components of the supports intersect. In two dimensions, such points correspond to junctions where the supports meet at equal angles. See \cite{TavarezTerracini} for similar results applied to a broader class of systems.
 We also refer to \cite{NV} for more refined stratification results on the singular set of stationary harmonic maps, and to \cite{Alper, EE} for generalizations of the techniques introduced in \cite{NV} to free boundary problems.


The system \eqref{adjacentsegregation}, when the Laplace operator is replaced by the fully nonlinear negative Pucci operator, has been studied by Quitalo \cite{V}.
Fully nonlinear diffusion operators, such as the Pucci operators, can be interpreted as modeling extremal diffusion mechanisms in which spreading occurs along preferred directions determined by the local curvature of the density, in particular along directions of maximal concavity or convexity. They also arise naturally in applications such as stochastic control,  where densities evolve along extremal directions determined by optimal strategies, see, for instance, \cite{Krylov}. From a mathematical point of view, these operators are studied precisely because standard variational or energy-based methods do not apply.

In \cite{V}, existence of positive solutions and convergence to  a limiting segregated configuration is established. In the case of two populations, Caffarelli, Quitalo, Patrizi, and Torres \cite{CL3} showed that the  limiting problem becomes a two-phase free boundary problem with the associated free boundary  condition 
$\frac {\partial u_1}{\partial \nu_1} = \frac {\partial u_2}{\partial \nu_2}$,  where $\nu_1$ and $\nu_2$ denote the interior normal directions to the respective supports.
 This formulation allowed for the application of  the sup-convolutions method, originally developed by Caffarelli in the linear setting, to prove that the regular points   form an open subset of the free boundary  locally of class $C^{1,\alpha}$. For a comprehensive discussion of the sup-convolution method and the general theory of free boundaries, the reader is referred to the monograph by Caffarelli and Salsa \cite{Geometric_Free_Boundary}. For  two-phase free boundary problems governed by fully nonlinear operators, we refer the reader to \cite{AF, DeSilvaSavin, DeSilvaFerrariSalsa} and the  references therein. 

In all the works mentioned above, the interaction between populations is adjacent, meaning that  $u_i(x)$ interacts with the other densities evaluated at the same point $x$. 
However, there are many processes where the growth of species $i$ 
 is inhibited by populations $j$  occupying an entire neighborhood around $x$, see for example \cite{CuMaMa, MiEiFang}.
As a first step in studying this nonlocal interaction, Caffarelli, Patrizi, and Quitalo \cite{CL2}  introduced system \eqref{main_problems} with the Laplace operator and $R=1$.
They proved existence of solutions and convergence to a limiting segregated configuration. 

 Unlike the case of adjacent segregation, here the species segregate at a distance of at least 1 from each other, with the distance exactly equal to 1 when  $H_1$ is defined as  in \eqref{H1} with $p=1$, or as in 
 \eqref{H2}.  This form of segregation ensures that species no longer interact in the limit.
   Moreover, under suitable assumptions on the boundary data, it was shown that in dimension 2 the free boundaries are Lipschitz curves, and the number of singular points (edges) is finite. 
   These edges arise at points that are exactly distance 1 from two or more other connected components of the supports. At such points, the edges of the free boundaries all coincide. A free boundary condition was also derived. 

 When $p=2$ in \eqref{H1},  minimizing solutions  and the limiting configurations of \eqref{main_problems} have been studied in \cite{NS1,NS2}.  
   

In the present paper, we study segregation at a distance governed by fully nonlinear diffusion, with long-range interaction at distance $R$, as modeled by system \eqref{main_problems}. We prove the existence of solutions (Theorem \ref{main theorem1}), and their convergence to a limiting segregated configuration (Theorem \ref{main theorem3}), and we begin the analysis of the geometric properties of the resulting free boundaries.   In particular, we show that the supports of the populations are sets of finite perimeter (Theorem \ref{fp})  and that the free boundaries satisfy a uniform exterior ball condition (Theorem \ref{main theorem5}). 

The proofs adapt arguments developed in \cite{CL2} and rely on the comparison principle, barrier arguments, and regularity results for uniformly elliptic operators. 
When $H_R$ is defined as in \eqref{H2}, the negative Pucci operator can be replaced by a more general uniformly elliptic operator, such as the positive Pucci operator.  When  $H_R$ is defined as in \eqref{H1}, instead, subharmonicity of solutions is employed (see, for instance, the proof of Lemma~\ref{decaylemmaprelim})). 
This property (see \eqref{MlessthanLap}) follows from the fact that the negative Pucci operator is the infimum over a class of linear operators that includes the Laplace operator. A deeper analysis of free boundaries in the linear case is carried out in \cite{CL2} using variational arguments. An extension of these results to the fully nonlinear case, where non-variational methods are required, will be addressed in a forthcoming paper. 

One of the main motivations for studying this problem is to gain a better understanding of the  adjacent interaction model as $R\to0^+$, namely the limiting free boundary problem of system \eqref{adjacentsegregation} with the Laplace operator replaced by the fully nonlinear Pucci operator. 
From this perspective,  the limiting free boundary problem of system \eqref{main_problems} can be viewed as a regularization of the adjacent free boundary problem, since in the former, the free boundaries satisfy an exterior ball condition.

\subsection{Main results}

We assume  the boundary data satisfy  
\begin{equation}    \label{fiassumption}
        f_i : (\partial \Omega)_{\leq R} \rightarrow \mathbb{R}, f_i \geq 0, f_i \neq 0, \quad \textnormal{$f_i$ is H\"older continuous,}
       \end{equation}
and that there is a constant $c > 0$ such that for any $x \in \partial \Omega \cap \text{supp } f_i$:
\begin{equation}
\label{size_of_ball_intersect_support}
|{B}_r(x) \cap \text {supp }f_i| \geq c|{B}_r(x)|,
\end{equation}
and
       \begin{equation}\label{f_i-f_j_disjointsupport}       
        d(\tn{supp} \, f_i, \tn{supp} \, f_j) \geq 1
       \end{equation}
for all $i \neq j$.

The density condition \eqref{size_of_ball_intersect_support}, although it may not appear natural a priori, guarantees segregation at distance of solutions up to the boundary of $\Omega$ (see the proof of Lemma \ref{seconddecaylemma}).  

Our first main result establishes existence of positive solutions to \eqref{main theorem1}: 
\begin{theorem}
[Existence of Solutions] \label{main theorem1}
  Let $\Omega$ be a bounded Lipschitz domain of $\R^n$.
Assume \eqref{fiassumption} holds true.
Then for any $\eps  > 0$, and $0<R\leq 1$, there exist positive functions 
$u^{\eps}_1,\ldots, u^{\eps}_K\in  C^{\alpha}(\overline{\Omega})\cap C_{loc}^{2,\alpha}(\Omega)$, for some $0<\alpha<1$, 
which  are  solutions of  problem \eqref{main_problems}.
\end{theorem}

Uniqueness of solutions to system \eqref{main_problems} was proved in \cite{Bozorgnia} for the linear case, with  $H_R$ defined as in \eqref{H1} and  $p=1$.

The next result concerns the limiting behavior of solutions as $\eps\to 0^+$.  As in the linear case,  
 the supports of the limiting functions remain separated by at least distance $R$.


\begin{theorem}
[Limit Problem]\label{main theorem3}
  Let $\Omega$ be a bounded Lipschitz domain in $\R^n$.  
Assume \eqref{fiassumption}, \eqref{size_of_ball_intersect_support}  and \eqref{f_i-f_j_disjointsupport}   hold true.
For any $\eps > 0$ and $0<R\leq 1$,  let $(u^{\eps}_1, \ldots, u^{\eps}_K)$ be  a solution to \eqref{main_problems}. Then there exists a subsequence of $\{(u^{\eps}_1,\ldots, u^{\eps}_K) \}_{\eps > 0}$
that converges locally uniformly in $\Omega$ to a limit function $(u_1, \ldots, u_K)$ as $\eps \rightarrow 0^+$. Moreover, the limit function $(u_1, \ldots, u_K)$
has the following properties:
\begin{enumerate}
 \item Each function $u_i$ is locally Lipschitz continuous on $\Omega$.
\item  $\mathcal{M}^-(u_i) = 0$ on $\{u_i > 0\}\cap\Omega$,  for any $i=1,\ldots, K$.
 \item  For any $1 \leq i < j \leq K$, the supports of the function $u_i$ and $u_j$ are at distance at least $\sizeofboundary$ from each other. That is, $u_i$ vanishes on the set \{$x \in \Omega$: $d(x, \textnormal{supp } u_j) \leq \sizeofboundary\}$    for any $i \neq j$.  
\end{enumerate}
\end{theorem}
We now turn to the regularity properties of the sets  $\{ u_i > 0\}\cap\Omega$ and their corresponding  free boundaries  $\partial \{ u_i > 0\}\cap\Omega$. 

Let $(u_1, \ldots, u_K)$ be a subsequential limit of  $\{(u^{\eps}_1, \ldots, u^{\eps}_K) \}_{\eps > 0}$ as in Theorem \ref{main theorem3}. The following two geometric properties hold:

\begin{theorem}
[A Semiconvexity Property of the Free Boundary]\label{main theorem5}
 If $x_0 \in \partial \{ u_i > 0\}\cap\Omega$ for some $i = 1, \ldots, K$, there is an exterior tangent ball of radius $\sizeofboundary$ at $x_0$.
\end{theorem}

\begin{theorem}
\label{fp}
For each $i = 1, \ldots, K$, the set $\{ u_i > 0\}\cap\Omega$ has finite perimeter. 
\end{theorem}

\subsection{Organization of the paper}

The paper is organized as follows. In Section 2, we present regularity results and comparison principles for fully nonlinear equations that 
that will be used throughout the paper.   In Section 3, we apply the Schauder fixed-point theorem to establish the existence of solutions to system  \eqref{main_problems} (Theorem \ref{main theorem1}). 
 Section 4 is devoted to proving exponential decay properties for  the functions $u_i^{\eps}$,  that are fundamental to let $\eps \rightarrow 0^+$  and obtain convergence to locally Lipschitz  limiting functions $u_i$ (Theorem \ref{main theorem3}). In Section 5,  we prove 
Theorems \ref{main theorem5} and \ref{fp}. 
Finally, in the Appendix, we provide a geometric property of the distance function that is used in  Section 5.

\subsection{Notations}

In the paper, we will denote by $C>0$ any constant depending only on the dimension $n$, the domain $\Omega$ and the boundary data $f_i$. 

We let $B_r(x_0)$  denote the ball of radius $r>0$ centered at $x_0\in \R^n$. 

 Given a subset $E$ of $\R^n$, we let  $\partial E$ and $|E|$ denote the topological boundary and the Lebesgue measure of the set respectively.


The support of a function $f$ is denoted as $\textnormal{supp } f$, and the average of $f$ over the set $E$ as $\fint_E fdx := \frac{1}{|E|}\int_E fdx.$ 

For $0<\alpha <1$ and $k \in \N$, we denote by $C^{k,\alpha}(E)$ the usual class of functions with bounded $C^{k,\alpha}$ norm over the domain $E\subset \R^n$. For $\alpha=0$ we simply write $C^{k}(E)$. For $k=0$ we simply write $C^{\alpha}(E)$. 

For $p>1$ and $k \in \N$, we denote by $W^{k,p}(E)$ the usual  Sobolev space of functions which are  $L^p(E)$ and whose  weak derivatives are in $L^p(E)$ up to order $k$.



\section{Preliminaries}\label{prelsection}
In this section, we  present some fundamental results for fully nonlinear elliptic equations that  will  be used throughout  this paper. 
We  begin by recalling the definition and some basic properties of Pucci  extremal operators. Next, we review some key regularity results and comparison principles for viscosity solutions of equations involving Pucci operators.
We then recall  the Schauder fixed-point theorem. Finally, we finish with a remark on the scaling property of the operator $H_R$.

 \subsection{The Pucci  extremal operators.}Let
 $0 < \lambda \leq \Lambda$ be given positive constants. For any $n\times n$ symmetric real matrix $M$,  the Pucci's extremal operators are defined by
\begin{eqnarray}
\label{def_of_M-}
\M^-(M) := \lambda \sum_{e_i > 0}e_i+ \Lambda \sum_{e_i < 0}e_i  \nonumber \newline \\
 \M^+(M):= \Lambda \sum_{e_i > 0}e_i+ \lambda \sum_{e_i < 0}e_i, \nonumber
\end{eqnarray}
where $e_i = e_i(M)$, $1 \leq i \leq n$, are the eigenvalues of the  matrix $M$.

 For a function $u$ of class $C^2$, we define 
\begin{equation}
\M^-(u) := \M^-(D^2u). 
\end{equation}
The operators $\M^-$ and $\M^+$ are fully nonlinear and uniformly elliptic. For a more comprehensive discussion on the general theory of fully nonlinear uniformly elliptic operators, we refer the reader to \cite{CL1}. Moreover, these operators enjoy the following basic properties:

\begin{lemma}
\label{basic_properties_Pucci}
Let $0 < \lambda \leq \Lambda$. For any  $n\times n$ symmetric matrices $M$ and $N$,  the following properties hold:
\begin{enumerate}  
    \item  $\M^-(M) = -\M^+(-M)$.
    \item  $\M^{\pm}(t M)=t\M^{\pm}(M)$, if $t\geq 0$.
    \item\label{3prop}  $\M^-(M)\leq \lambda \text{tr}M$ and $\M^+(M)\geq \Lambda \text{tr}M$.
   \item\label{4prop} $\M^-(M)+\M^-(N) \leq \M^-(M+N) \leq \M^-(M)+\M^+(N)$.
    \item $\M^+(M)+\M^-(N) \leq \M^+(M+N) \leq \M^+(M)+\M^+(N)$.
\end{enumerate}
\end{lemma}
Property  \eqref{3prop} of Lemma \ref{basic_properties_Pucci} implies  that for any $C^2$ function $u$, we have 
\begin{equation}\label{MlessthanLap}\M^-(u)\leq \lambda \Delta u,\end{equation}
which will be used several times later in the paper.

 \subsection{Regularity results}In this subsection, we recall some well-known regularity results for viscosity solutions of equations involving Pucci  operators. For the definition of a viscosity solution, we refer the reader to \cite{CL1,IL}.   We begin with the Harnack inequality for Pucci operators: 
 
\begin{theorem}\label{Harnackthm}
\cite[Theorem 4.3]{CL1}
{\em (Harnack Inequality)}{\bf .}
\label{Harnack_Inequality}
Let $u \in  C(B_1(0))$ be non-negative in $B_1(0)$ and satisfy in the viscosity sense 
$$\M^+u\geq -|g|\quad \text{and} \quad \M^- u\leq |g|\quad \text{in }  B_1(0),$$ where $g$ is a bounded continuous function in $ B_1(0)$. Then 
\begin{eqnarray*}
    \sup_{ B_{\frac12}(0)} u \leq C (\inf_{ B_{\frac12}(0)} u+\norm{g}_{L^n( B_1(0)})),
\end{eqnarray*}
where $C>0$ is a universal constant.
\end{theorem}

\begin{theorem}
\cite[Theorem 7.1]{CL1}
{\em (Interior $W^{2,p}$ Regularity of Viscosity Solutions)}{\bf .}
\label{W2pRegularity} 
Let $u \in  C(B_1(0))$ be a viscosity solution of  $\M^-(D^2u) = g$ in $B_1(0)$, with $g\in L^p(B_1(0))$, $n<p<\infty$. Then $u\in W^{2,p}(B_{\frac12}(0))$ and 
 $$\norm{u}_{W^{2,p}( B_{\frac{1}{2}})} \leq C(\norm{u}_{C(B_1(0))}+\|g\|_{L^p(B_1(0))}),$$
 where  $C>0$ is a universal constant. 
\end{theorem}

The following result combines interior $C^{2,\alpha}$ regularity with $C^\alpha$ regularity up to the boundary. For the interior $C^{2,\alpha}$ regularity for equations with  Pucci operators and  a $C^\alpha$ right-hand side, we refer to \cite[Theorem 8.1]{CL1}. For the $C^\alpha$ regularity up to the boundary in Lipschitz domains, see  \cite[Theorem VII.1]{IshiiLions} and \cite[Proposition 3.10]{V}. 
\begin{theorem}
\label{CalphaRegularityLip} 
Let $\Omega$ be a bounded Lipschitz domain. Assume $g\in C(\Omega)$ and  $f\in C^\beta(\partial\Omega)$, 
 for some $0<\beta<1$. Let $u\in C(\overline\Omega)$ be the viscosity solution of 
$$\begin{cases}
\M^-u=g&\text{in }\Omega,\\
u=f&\text{on }\partial\Omega. 
\end{cases}
$$
Then, there exists $0<\alpha\leq \beta$ such that $u\in  C^{\alpha}(\overline\Omega)$. If in addition $g\in C^\gamma(\Omega)$, for some $0<\gamma<1$, then there exists $0<\alpha\leq \gamma$ such that $u\in C_{loc}^{2,\alpha}(\Omega)$. 
\end{theorem}

\subsection{The comparison principle}
We will frequently use the following comparison principle throughout the paper, for which we refer the reader to \cite[Theorem 3.6]{CL1}. 
 \begin{theorem} [Comparison Principle]
Let $\Omega$ be a bounded domain of $\R^n$. Let $a,\,g\in C(\Omega)$ with $a\geq 0$ in $\Omega$. Assume that  $u,\,v\in C(\overline\Omega)$ satisfy in the viscosity sense 
$$\M^-u\leq a(x)u+g(x)\quad\text{and}\quad\M^+v\geq a(x)v+g(x)\quad \text{in }\Omega.$$
If $u\geq v$ on $\partial \Omega$, then $u\geq v$ in $\Omega$. 
 \end{theorem}
 The following minimum principle is an immediate corollary of the above comparison theorem. 
 \begin{theorem} [Minimum Principle]
Let $\Omega$ be a bounded domain  of $\R^n$. Let $a\in C(\Omega)$ with $a\geq 0$ in $\Omega$. Assume that  $u\in C(\overline\Omega)$ satisfy in the viscosity sense 
$$\M^-u\leq  a(x)u\quad \text{in }\Omega.$$
If $u\geq 0$ on $\partial \Omega$, then $u\geq 0$ on $\Omega$. 
 \end{theorem}
 We  conclude this subsection by recalling the so-called strong minimum  principle (\cite[Proposition 4.9]{CL1}): 
\begin{theorem}[Strong Minimum Principle]
\label{strong_maximum_principle}
Let $\Omega$ be a domain of $\R^n,$ and assume $a\in C(\Omega)$. 
 Let $u \in C(\overline\Omega) $ satisfy in the viscosity sense $$\M^-u\leq  a(x)u\quad \text{in }\Omega.$$ Assume that $u \geq 0$ in $\Omega$ and $u(x_0) = 0$ for 
some $x_0 \in \Omega$, then $u$ vanishes identically in $\Omega$.
 \end{theorem}
 A consequence of the above theorems is that if $\Omega$ is a bounded domain and $u$ is a continuous function on $\Omega$ satisfying $\Delta u \geq 0$ in the viscosity sense, then
$\sup_{\Omega} u = \sup_{\partial \Omega} u$.
This result will be used several times later in the paper.

\subsection{The Schauder fixed-point theorem}
The following fixed-point theorem will be used to prove the existence of a solution to \eqref{main_problems}.
\begin{theorem}(\cite[Corollary 11.2]{GT})
\label{fixedpointthm}
  Let $B$ be a nonempty closed convex subset of a real Banach space $X$, and let $T: B \rightarrow B$ be continuous on $B$ such that $T(B)$ is precompact. Then T has a fixed point.
\end{theorem}

\subsection{Scaling properties of $H_R$.}
Let $H_R$ be defined as in \eqref{H1} or  \eqref{H2}. Let us define $w_{R,x}(y)=w(x+R(y-z))$. Then, we have the following scaling property: 
\begin{equation}
\label{HRscaling}
H_1(w_{R,x})(z)=H_R(w)(x).
\end{equation}

\section{Existence of Solutions for the Model  \eqref{main_problems}}

Using Theorem \ref{fixedpointthm}, we can prove Theorem \ref{main theorem1}.

{\bf Proof of Theorem \ref{main theorem1}:}
We adapt the proof in \cite[Theorem 4.1]{CL2} to the nonlinear case with the Pucci operator, see also \cite[Theorem 2.2]{V}. Let $X$$:=$ $\{u = (u_1, \ldots, u_K): \overline\Omega \rightarrow   \mathbb{R}^K$: $u$ is continuous on $\overline\Omega$\}, which is a real Banach space under the sup-norm defined by
$\norm{u}_{\infty}$:= $\max_{1\leq j \leq K}$$\norm{u_j}_{\infty}$.

For each $1\leq i \leq K$, let $\phi_ i$ to be the (unique) viscosity solution of the problem
\begin{equation}
\label{mainequations}
\begin{cases} \M^-(\phi_i)= 0\quad & \text{ in } \Omega, \\
        \phi_i= f_i & \text{ on } \partial \Omega,
       \end{cases}
\end{equation}
 whose existence is guaranteed by Perron's method (\cite[Theorem 4.1]{IL}).

Next, consider $B$ $:=$ $\{u = (u_1, \ldots, u_K) \in X$: $0 \leq  u_i \leq \phi_i$ in $\Omega$, $u_i = f_i$ on $(\partial \Omega)_{\leq \sizeofboundary}$, $1 \leq i \leq K$\}. Clearly, $B$ is a nonempty closed convex subset of $X$. We define a mapping $T^{\eps}$: $B \rightarrow B$  by  $T^{\eps}(u_1, \ldots, u_K)$ $=$ $(v^{\eps}_1, \ldots, v^{\eps}_K)$ if and only if $(v^{\eps}_1,\ldots, v^{\eps}_K)$ is the viscosity solution of the problem
\begin{equation}
\label{mainequations2.2}
\begin{cases} \M^-(v^{\varepsilon}_{i})= \frac{1}{\eps^2} v^{\eps}_i  \sum_{j \neq i}  H_R(u_j)(x)\quad & \text{ in } \Omega, \\
        v^{\eps}_i= f_i & \text{ on } (\partial \Omega)_{\leq \sizeofboundary}.
       \end{cases}
\end{equation}
 {Note that the zero-order term in \eqref{mainequations2.2} is non-negative, so the comparison principle and consequently Perron's method apply. }
The proof of the existence of a viscosity solution of \eqref{main_problems} is completed if we show  that the mapping  $T^{\eps}$ has a fixed point. To this end, by 
Theorem \ref{fixedpointthm},
it suffices to prove the following:
\begin{enumerate}
\item $T^{\eps}$ maps $B$ into $B$.
\item $T^{\eps}$ is continuous on $B$.
\item $T^{\eps}(B)$ is precompact.
\end{enumerate}

{ 
\medskip
 \noindent \textbf{(1) 
$T^{\eps}$ maps $B$ into $B$:} 
Let} $(u_1, \ldots, u_K) \in B$, and  $(v^{\eps}_1, \ldots, v^{\eps}_K)$ $=$ $T^{\eps}(u_1, \ldots, u_K)$. We need to check that $(v^{\eps}_1, \ldots, v^{\eps}_K) \in B$.
{Fix any $1\leq i\leq K$. Since $v^\eps_i=f_i\geq 0$ on $\partial\Omega$, by the maximum principle,  we have $v^\eps_i\geq 0$ in $\Omega$. 
In particular,   $v^\eps_i$ satisfies in the viscosity sense $\M^-  (v^\eps_i)\geq 0$ in $\Omega$. Therefore,  the comparison principle implies that $v^\eps_i\leq \phi_i$ in $\Omega$. This shows that 
$(v^{\eps}_1, \ldots, v^{\eps}_K) \in B$. }

{ 
\medskip
 \noindent \textbf{(2) $T^{\eps}$ is continuous on $B$:}
 Let  }$\{$($u_{1m}, \ldots, u_{Km})$\}$_{m \in \mathbb{N}}$  $\subset$  $B$ be such that $(u_{1m}, \ldots, u_{Km})$ $\rightarrow$  $(u_1, \ldots, u_K)$
{uniformly in $\Omega$}  as $m$ $\rightarrow \infty$.
Denote $(v^{\eps}_1, \ldots, v^{\eps}_K)=T^{\eps}(u_1, \ldots, u_K)$ and  $(v^{\eps}_{1m}, \ldots, v^{\eps}_{Km})=T^{\eps}(u_{1m}, \ldots, u_{Km})$,  for all $m \in \mathbb{N}$. 
We need to check that $(v^{\eps}_{1m}, \ldots, v^{\eps}_{Km})$ $\rightarrow$  $(v^{\eps}_1, \ldots, v^{\eps}_K)$ {uniformly in $\Omega$} as $m$ $\rightarrow \infty$. To see this, it suffices to show that there is a constant $C_\eps > 0$ (which depends on $\eps$, {$\Omega$, $n$ and the boundary data}), such that, for all $m \in \mathbb{N}$  and for all $ 1\leq i \leq K$,
\begin{equation}\label{mainidentityThm1.1}
\norm{v^{\eps}_{im}-v^{\eps}_{i}}_{\infty}\leq C_\eps\max_{1\leq j \leq K}\norm{u_{jm}-u_{j}}_{\infty}.
\end{equation} 
Fix  $r > 0$ large enough so that $\Omega \subset B_r(0)$.  For any $m \in \mathbb{N}$ we consider the function
 \begin{equation*}
 h_m(x):= D \max_{1\leq j \leq K}\norm{u_{jm}-u_{j}}_{\infty}(r^2-|x|^2), 
\end{equation*}
 where $D > 0$ is a constant to be chosen later. {Since $h_m$ is smooth, by Lemma \ref{basic_properties_Pucci}, we see that,  for all $1\leq i\leq K$, the function $h_m+v^\eps_i$ satisfies in the viscosity sense,
 \begin{equation}
 \label{hm+vepsequ}\begin{split}
\M^-(h_m+v^\eps_i)&\leq \M^+(h_m)+\M^-(v^\eps_i)\\&
=-2n\lambda D \max_{1\leq j \leq K}\norm{u_{jm}-u_{j}}_{\infty}+ \frac{1}{\eps^2} v^{\eps}_i  \sum_{j \neq i}  H_R(u_j)(x)\\&
=-2n\lambda D \max_{1\leq j \leq K}\norm{u_{jm}-u_{j}}_{\infty}+ \frac{1}{\eps^2} v^{\eps}_i  \sum_{j \neq i} \left( H_R(u_j)(x)- H_R(u_{jm})(x)\right)\\&+\frac{1}{\eps^2} v^{\eps}_i  \sum_{j \neq i} H_R(u_{jm})(x).
 \end{split}
 \end{equation}
 }
 We claim that there is a constant $C > 0$ such that, for all $x \in \Omega$
\begin{equation}
\label{equation3_minor}
\sum_{j \neq i}  \left( H_R(u_j)(x)- H_R(u_{jm})(x)\right)\leq  (K-1)C\max_{1\leq j \leq K}\norm{u_{jm}-u_{j}}_{\infty}.
\end{equation}
To see this, note that if $H_R$ is given by \eqref{H2}, we have, for any $j \neq i$ and $y \in {B}_\sizeofboundary(x)$, 
$u_j(y)-u_{jm}(y) \leq \max_{1\leq j \leq K}\norm{u_{jm}-u_{j}}_{\infty}$, so that $u_j(y) \leq \max_{1\leq j \leq K}\norm{u_{jm}-u_{j}}_{\infty}+u_{jm}(y)$. 
Taking  the supremum over all  $y \in {B}_\sizeofboundary(x)$  and then summing over over all $j \neq i$ yields \eqref{equation3_minor} with $C=1$.

Now assume that  $H_R$ is given by \eqref{H1}. We then have,   for  $j \neq i$, 
\begin{eqnarray*}
H_R(u_j)(x)-H_R(u_{jm})(x) &= & \fint_{{B}_\sizeofboundary(x)}(u_j^p(y)-u_{jm}^p(y))\,dy  \\
&\leq & \fint_{{B}_\sizeofboundary(x)}p{\bar C}^{p-1}|u_j(y)-u_{jm}(y)|\, dy \\
&\leq&  p{\bar C}^{p-1}\max_{1\leq j \leq K}\norm{u_{jm}-u_{j}}_{\infty},  \\
\end{eqnarray*}
where $\bar C = \max_{1\leq j \leq K}\{\sup_{m\in\N} \norm{u_{jm}}_{\infty}, \norm{u_{j}}_{\infty}\}<\infty.$
Taking the summation  over  $j \neq i$, we obtain \eqref{equation3_minor} with  $C= p{\bar C}^{p-1}$. From \eqref{hm+vepsequ} and  \eqref{equation3_minor}  we infer that, choosing $D=D_\eps$ such that 
$$D\geq \frac{\max_{1\leq i\leq K}\|\phi_i\|_\infty(K-1)C}{2n \lambda \eps^2},$$
the function  $h_m+v^\eps_i$ satisfies in the viscosity sense
\begin{align*}
\M^-(h_m+v^\eps_i)&\leq \frac{1}{\eps^2} v^{\eps}_i  \sum_{j \neq i} H_R(u_{jm})(x)\leq \frac{1}{\eps^2} (h_m+v^{\eps}_i)  \sum_{j \neq i} H_R(u_{jm})(x). 
\end{align*}
Since in addition $h_m+v^\eps_i\geq v^\eps_i=f_i= v^\eps_{im}$ on $\partial\Omega$, by the comparison principle we have  $h_m+v^\eps_{i} \geq v^\eps_{im}$ in $\Omega$, which implies that for all $x\in\Omega$, 
$$v^\eps_{im}(x)-v^\eps_{i}(x)\leq h_m(x)\leq r^2D\max_{1\leq j \leq K}\norm{u_{jm}-u_{j}}_{\infty}. $$
Similarly, one can prove that $v^\eps_{i}-v^\eps_{im}\leq r^2D\max_{1\leq j \leq K}\norm{u_{jm}-u_{j}}_{\infty} $ in $\Omega$.  Estimate \eqref{mainidentityThm1.1} follows. This concludes the proof of (2).

\medskip
 \noindent \textbf{(3) $T^{\eps}(B)$ is precompact:} Let  $\{$($u_{1m},\ldots, u_{Km})$\}$_{m \in \mathbb{N}}$ $\subset$ $B$ be a bounded sequence in $B$   and $(v^{\eps}_{1m},\ldots, v^{\eps}_{Km})$$=$  $T^{\eps}(u_{1m},\ldots, u_{Km}) \in B$,  for  $m \in \mathbb{N}$. By  Theorem \ref{CalphaRegularityLip},  
 there is  $0<\alpha<1$ such that $\{$($v^{\eps}_{1m}, \ldots, v^{\eps}_{Km})$\} is bounded in $C^{\alpha}(\overline \Omega; \mathbb{R}^K)$. By the compact embedding  of $C^{\alpha}(\overline\Omega;\mathbb{R}^K)$ in $X$, there is a subsequence of $\{(u_{1m},\ldots, u_{Km})\}$ that converges in $B$. This shows that $T^{\eps}(B)$ is precompact.

\medskip 

By (1)-(3) and  Theorem \ref{fixedpointthm},  $T^{\eps}$
has a fixed point in $B$. This completes the proof of the existence of a viscosity solution $(u^\eps_1,\ldots,u^\eps_K)$ to  \eqref{main_problems}. 
Moreover, by the strong minimum principle, each function $u_i^\eps$ is positive in $\Omega$.

By Theorem \ref{CalphaRegularityLip} and \eqref{fiassumption}, $u_i^\eps\in C^\alpha(\Omega\cup  (\partial \Omega)_{\leq {\sizeofboundary}})$ for all $1\leq i\leq K$ and some $0<\alpha<1$. 
 Let us show that this implies that  $H_R(u^\eps_i)\in C^\alpha(\Omega)$  for all $1\leq i\leq K$.  Let $x_1,\,x_2\in\Omega$. First, assume that $H_R$ is defined  as in \eqref{H1}, then 
\begin{align*}
|H_R(u^\eps_i)(x_1)-H_R(u^\eps_i)(x_2)|=&\left|\fint_{B_R(0)}[(u^\eps_i)^p(x_1+y)-(u^\eps_i)^p(x_2+y)]\,dy\right|\\&
\leq p\|u^\eps_i\|^{p-1}_\infty \fint_{B_R(0)}|u^\eps_i(x_1+y)-u^\eps_i(x_2+y)|\,dy\\&\leq C_\eps|x_1-x_2|^\alpha,
\end{align*}
 as desired. 
 
 Next, assume   that $H_R$ is defined  as in \eqref{H2}. Let $z_1\in\overline{ B_R(x_1)}$ be such that $H_R(u^\eps_i)(x_1)=u^\eps_i(z_1)$ and define $z_2:=z_1+x_2-x_1$. Note that 
 $|z_2-x_2|\leq  R$, so $z_2\in\overline{ B_R(x_2)}$, and hence 
 \begin{align*}
H_R(u^\eps_i)(x_1)-H_R(u^\eps_i)(x_2)&\leq u^\eps_i(z_1)-u^\eps_i(z_2)\leq C_\eps|z_1-z_2|^\alpha=C_\eps|x_1-x_2|^\alpha.  
\end{align*}
 This implies that $H_R(u^\eps_i)\in C^\alpha(\Omega)$.

Since the product of functions in $C^\alpha(\Omega) $ belongs to  $C^\alpha(\Omega)$, we deduce that  the right-hand side of equation   \eqref{main_problems} lies in this space.
 Applying Theorem \ref{CalphaRegularityLip} once more, we conclude  that $u^\eps_i\in C^\alpha(\overline\Omega )\cap C^{2,\alpha}_{loc}(\Omega)$ for all $1\leq i\leq K$ and some $0<\alpha<1$. $\Box$

\section{Uniform Estimates and the Limit Problem}

This section is dedicated to the proof of  Theorem \ref{main theorem3}. We will show that a subsequence of $\{(u^{\eps}_1,\ldots, u^{\eps}_K)\}_{\eps > 0}$ converges uniformly on compact subsets of $\Omega$ to a limit function $(u_1,\ldots, u_K)$ by establishing a Lipschitz estimate that is uniform in $\eps$. To this end, we will prove that each  function $u^\eps_i$ exhibits  exponential decay in  neighborhoods of size $R$ around  the regions where the other  functions $u^\eps_j$, $j\neq i$, stay away from zero,  see Lemmas \ref{firstdecaylemma} and \ref{seconddecaylemma} below. This  will ensure that the right-hand side of \eqref{main_problems} go to 0  uniformly as $\eps\to 0^+$ in those regions. 
We begin by stating two lemmas that are fundamental to  establishing this exponential decay.

\begin{lemma}
\label{subharmoniclemma}
(\cite[Lemma 5.1]{CL2})
    Let $\omega$ be a subharmonic function in $ {B}_1(0)$ such that $\omega \leq 1$ in ${B}_1(0)$ and $\omega(0) = m > 0$. Also, let $D_0 \subset \mathbb{R}^n$ be a smooth  domain with curvatures bounded by a positive constant $C_0$.
    Then there exists a universal constant $ \tau_0=\tau_0(n,C_0)>0$ such that if $d(D_0, 0) \leq \tau_0 m$, then $\sup_{\partial D_0\cap {B}_1(0)} \omega \geq \frac{m}{2}$. 
\end{lemma}

\begin{lemma}
\label{decaylemma}
(\cite[Lemma 5.2]{CL2})
    Let $\omega\in C(B_r(0))$ be a positive viscosity subsolution of the linear uniformly elliptic equation $a_{ij}D_{ij}\omega = \theta^2\omega$ in $B_r(0)$. Then there exist two constants $c, C > 0$ such that $\frac{w(0)}{\sup_{B_r(0)} w} \leq Ce^{-c\theta r}$. 
\end{lemma}

\noindent  We will  also use the following result whose proof follows \cite[Lemma 5.3]{CL2}.

\begin{lemma} \label{decaylemmaprelim} 
 Let $0 < r < 1$, and let $\omega$ be a function satisfying the following conditions in $B_{2r}(0)$: 
\begin{enumerate}
 \item $0 \leq \omega \leq 1$; 
 \item $\omega$ is subharmonic; 
 \item $\omega(0) = m$. 
 \end{enumerate} 
 Then there exists a universal constant $0 < \tau < 1$ such that, if $|\bar{x}| \leq 1 + \tau m r / 2$, we have for any $x \in B_{\frac{\tau m}{4}}(\bar{x})$, 
 
 \begin{equation}\label{H1omega_case1} H_1(\omega)(x) \geq \frac{m}{2} \quad \text{if  $H_1$ is defined as in \eqref{H2}},
 \end{equation} 
 and \begin{equation}\label{H1omega_case2} H_1(\omega)(x) \geq C m^{p+n} r^n \quad \text{if $H_1$ is defined as in \eqref{H1},} \end{equation} for some constant $C > 0$ depending on $p$, $\tau$, and $n$.
 
\end{lemma}

\begin{proof}

Let $\tau$ be  the  constant $\tau_0$ given by  Lemma \ref{subharmoniclemma} with $C_0=2.$ Without loss of generality, we may assume $\tau<1$. Let $\bar x$ be such that $|\bar x|\leq 1+\tau mr/2$. 
We will apply the lemma to the rescaled function $\tilde \omega(x)= \omega (rx)$ with $D_0=B_{\frac{1}{r}-\frac{\tau m }{2}}(\frac{\bar x}{r}).$ Note that $D_0$ is at distance less or equal to $\tau m$ from the origin  and its principal curvatures, 
 all equal to $\frac{1}{\frac{1}{r}-\frac{\tau m }{2}}$, are  bounded  by
$$
\frac{1}{\frac{1}{r}-\frac{\tau m }{2}}=\frac{2r}{2-\tau mr}\leq 2.
$$
By Lemma \ref{subharmoniclemma},   there exists a point  $\bar z\in B_r(0)\cap \partial B_{1-\frac{\tau m r}{2}}(\bar x)$ such that  $\omega(\bar z)\geq m/2$. 
Now,  let $x\in B_{\frac{\tau m r}{4}}(\bar x)$. Then 
$$|x-\bar z|\leq |x-\bar x|+|\bar x-\bar z|\leq 1-\frac{\tau m r}{4}.$$
In particular, $B_\frac{\tau mr}{8}(\bar z)\subset B_1(x). $ Since $\bar z\in B_r(0)$, we also have $B_\frac{\tau mr}{8}(\bar z)\subset B_{2r}(0)$. 

First, consider  the case where $H_1$ is defined as in  \eqref{H2}.  Then, 
\begin{equation*}
\label{lower_bound_for_H_in_case2}
    H_1(\omega)(x) = \sup_{{B}_1(x)}\omega \geq \omega (\bar z) \geq \frac{m}{2}. 
\end{equation*}
This proves \eqref{H1omega_case1}. 

Next, assume $H_1$ as in \eqref{H1}. Since $p\geq 1$, $w^p$ is still subharmonic in $B_{2r}(0)$. 
Then, by  the mean value formula (see \cite[Theorem 2.1]{GT}) applied   in  $B_\frac{\tau mr}{8}(\bar z)\subset B_{2r}(0)$, we obtain
\begin{eqnarray*}
\label{lower_bound_for_H_in_case1}
H_1(\omega)(x) \nonumber 
&=& \fint_{{B}_1(x)}\omega^p(y)\, dy \nonumber 
\geq \frac{1}{|{B}_{1}(x)|}\int_{B_{\frac{\tau mr}{8}}(\bar z)}\omega^p(y)\, dy 
= \frac{\tau^nm^nr^n}{8^n}\fint_{B_{\frac{\tau mr}{8}}(\bar z)}\omega^p(y)\, dy\nonumber\\ 
&\geq& \frac{\tau^nm^nr^n}{8^n}\omega^p(\bar z) \nonumber
\geq   \frac{\tau^nm^{n+p}r^n}{2^p8^n}.    
\end{eqnarray*}
This proves \eqref{H1omega_case2} and completes the proof of the lemma.
 
 \end{proof}

Following \cite[Lemma 5.3]{CL2}, we now prove the exponential decay of the functions $u^\eps_j$ away from the boundary of $\Omega$, and at distances less than $R$ from the support of $u^\eps_i$,   $i \neq j$.

\begin{lemma}
\label{firstdecaylemma}    
    Let $(u^{\eps}_1, \ldots , u^{\eps}_K)$ be a solution of .  For  $i = 1, \ldots, K$, $0 < r <1$  and $\sigma > 0$, let  
    $$\Gamma^{\sigma, r}_i :=  \{ x \in \Omega\,: \,d(x, \textnormal{supp }f_i) \geq 2Rr,\, u^{\eps}_i (x)= \sigma\},$$ and  $$m := \frac{\sigma}{\sup_{\partial \Omega} f_i}.$$ Then, for $\tau$ defined as in Lemma \ref{decaylemmaprelim},  in the sets 
    \begin{eqnarray*}
        \Sigma^{\sigma, r}_{i,j} := \left\{x \in \Omega\,:\, d(x, \Gamma^{\sigma, r}_i) \leq R+\frac{\tau mRr}{2},\, d(x, \textnormal{supp }f_j) \geq \frac{\tau mRr}{4}\right\},
    \end{eqnarray*}
    we have 
    $u^{\eps}_j \leq C e^{\frac{-c\sigma^{\alpha}r^{\beta}R}{\eps}}$ for all $j \neq i$ and some $C, c,\,\alpha,\,\beta>0$ depending on $p,\tau,n$ and the ellipticity constants. 
\end{lemma}
\begin{proof}
Let $\bar x \in \Sigma^{\sigma, r}_{i,j}$. We claim that, for $j\neq i$,    
\begin{equation}
\label{mainequations5.0}
\Delta u^{\eps}_j \geq \frac{C \sigma^{\bar \alpha} r^{\bar \beta}}{\eps^2}u^{\eps}_j \quad \text{ in } {B}_{\frac{\tau mRr}{4}}(\bar x),
\end{equation} where $C,\,\bar \alpha>0$ and $\bar \beta\geq 0$  are constants depending on $p,\tau,n$ and $\lambda$. 
 Assuming \eqref{mainequations5.0}, we can apply Lemma \ref{decaylemma} to obtain
   $$u^{\eps}_j(\bar x) \leq C e^{- \frac{\bar{c}\sigma^{\frac{\bar\alpha}{2}}r^{{\frac{\bar\beta}{2}}}}{\eps} \frac{\tau mRr}{4}} = Ce^{-\frac{c\sigma^{\alpha}r^{\beta}R}{\eps}},$$ 
   where $\alpha = \bar \alpha/2+1$ and $\beta = \bar\beta/2+1$.  This completes the proof of the lemma, provided that \eqref{mainequations5.0} holds.
   
   To verify \eqref{mainequations5.0}, 
note that since $d(\bar x, \text{supp }f_j)\geq \tau mRr/4$, the ball ${B}_{\frac{\tau mRr}{4}}(\bar x)$ does not intersect $\text{supp }f_j$. Therefore, 
 recalling \eqref{MlessthanLap} and observing that $u^{\varepsilon}_j = 0$ in ${B}_{\frac{\tau mRr}{4}}(\bar x)\cap \Omega^c$, we have 
\begin{equation}
\label{Laplace_lowerbound}
\lambda \Delta u^{\eps}_j \geq \M^-(u^{\eps}_j) \geq \frac{1}{\eps^2} u^{\eps}_j  \sum_{k \neq j}  H_R(u^{\eps}_k)\geq \frac{1}{\eps^2} u^{\eps}_j  H_R(u^{\eps}_i)
\quad \text{ in } {B}_{\frac{\tau mRr}{4}}(\bar x).
\end{equation}
 We now estimate $H_R(u^{\eps}_i)$ in $B_{\frac{\tau mRr}{4}}(\bar x)$.  Let $\bar y\in \Gamma^{\sigma, r}_i$ satisfy   $|\bar x-\bar y|\leq R+\tau m Rr/2. $ Since $d(\bar y, \text{supp }f_i)\geq 2Rr$, $u^\eps_i$
 (extended by zero in ${B}_{2Rr}(\bar y)\cap \Omega^c$) satisfies $\Delta u^\eps_i\geq 0 $ in ${B}_{2Rr}(\bar y)$. Moreover, since $u^\eps_i$ 
is subharmonic in $\Omega$, it attains its maximum at the boundary of $\Omega$,   so that $\frac{u^\eps_i}{\sup_{\partial \Omega}f_i }\leq 1$ in $\Omega$.  
Define, 
$$\omega(y):=\frac{u^\eps_i(\bar y+\sizeofboundary  y)}{\sup_{\partial \Omega}f_i}. $$
Then, $0\leq \omega\leq 1$, $\omega(0)= \frac{u^\eps_i(\bar y)}{\sup_{\partial \Omega}f_i}=m$, $\Delta \omega\geq 0$ in ${B}_{2r}(0)$.
Since $|(\bar x-\bar y)/R|\leq 1+\tau mr/2$, by Lemma \ref{decaylemmaprelim},
$$H_1(\omega)\left(x\right)\geq C m^{\overline{\alpha}}r^{\overline{\beta}} \quad\text{for any }x\in {B}_{\frac{\tau mr}{4}}\left(\frac{\bar x-\bar y}{R}\right),$$
where $\overline{\alpha}=1$ and $\overline{\beta}=0$ if $H_1$ is defined as in \eqref{H2}, $\overline{\alpha}=p+n$ and $\overline{\beta}=n$ if $H_1$ is defined as in \eqref{H1}.
  Recalling  the scaling property \eqref{HRscaling}, we obtain
$$H_R(u^\eps_i)(x)\geq  C m^{\overline{\alpha}}r^{\overline{\beta}} \quad\text{for any }x\in {B}_{\frac{\tau mRr}{4}}\left(\bar x\right).$$
Inequality \eqref{mainequations5.0} follows. The proof of the lemma is thus completed. 
  \end{proof}

The following  result, whose proof follows \cite[Lemma 5.5]{CL2}, states that for any $i \neq j$, the function $u^{\eps}_i$ decays  exponentially to 0 in a strip  of size  $R$ around the support of $f_j$.

\begin{lemma}
\label{seconddecaylemma}   
    Let $(u^{\eps}_1, \ldots, u^{\eps}_K)$ be a solution of  problem \eqref{main_problems}. For  $j = 1, \ldots , K$, $\sigma > 0$, and $0 < r <\sizeofboundary$, let
     $$\bar{\Gamma}_j^{\sigma}:=\{x\in (\partial\Omega)_{\leq R}\,:\, f_j(x) \geq \sigma\} \subset \Omega^c.$$  
    Then on the sets  $\{x \in \Omega\, :\, d(x, \bar{\Gamma}_j^{\sigma} ) \leq \sizeofboundary-r$\}, we have 
    $u^{\eps}_i \leq C e^{\frac{-c\sigma^{\alpha}r^{\beta}}{\eps}}$ for all $i \neq j$,  for some $C,c,\alpha,\beta>0$ depending on $p$, $n$, the ellipticity constants
        and the modulus of continuity of $f_j$.
\end{lemma}
\begin{proof}
 Let $\bar x$ $\in$ $\{x \in \Omega\,:\,d(x, \bar{\Gamma}_j^{\sigma} ) \leq \sizeofboundary-r$\}. 
 We claim that  for $i\neq j$,

\begin{equation}
\label{mainequations5.2}
\Delta u^{\eps}_i \geq \frac{C \sigma^{\bar \alpha}  r^{\bar \beta}}{\eps^2}u^{\eps}_i \quad\text{in }{B}_{\frac{r}{2}}(\bar x),
\end{equation}
 where $C,\,\bar \alpha>0$ and $\bar \beta\geq 0$ are constants depending on $p, n, \lambda$, and the modulus of continuity of $f_j$. 
 Assuming \eqref{mainequations5.2}, we can apply Lemma \ref{decaylemma} to obtain  
   $$u^{\eps}_i(\bar x) \leq C e^{- \frac{\bar{c}\sigma^{\frac{\bar\alpha}{2}}r^{{\frac{\bar\beta}{2}}}}{\eps} \frac{r}{2}} = Ce^{-\frac{c\sigma^{\alpha}r^{\beta}}{\eps}},$$
    where  $\alpha = \bar \alpha/2+1$ and $\beta = \bar\beta/2+1$.   This completes the proof of the lemma, provided that \eqref{mainequations5.2} holds.

 To verify \eqref{mainequations5.2}, note that  by assumption \eqref{f_i-f_j_disjointsupport}   the ball $B_{\frac{r}{2}}(\bar x)$ does not intersect $\text{supp }f_i$. 
   Therefore,   recalling \eqref{MlessthanLap} and observing that $u^{\varepsilon}_i = 0$ in $B_{\frac{r}{2}}(\bar{x}) \cap \Omega^c$, we have 
    \begin{equation}
\label{Laplace_lowerbound_fi_lem}
\lambda \Delta u^{\eps}_i \geq \M^-(u^{\eps}_i) \geq \frac{1}{\eps^2} u^{\eps}_i  \sum_{k \neq i}  H_R(u^{\eps}_k)\geq \frac{1}{\eps^2} u^{\eps}_i H_R(u^{\eps}_j)
\quad \text{ in } {B}_{\frac{r}{2}}(\bar x).
\end{equation}
We now estimate $H_R(u^{\eps}_j)$ in $B_{\frac{r}{2}}(\bar x)$. Let  $\bar y \in$ $\bar{\Gamma}_j^{\sigma}$ satisfy  $|\bar x-\bar y| \leq \sizeofboundary-r$, and let 
 $x \in {B}_{\frac{r}{2}}(\bar x)$. Then, $|x-\bar y|\leq R-r/2$. 

First, consider the case where $H_R$ is given by \eqref{H2}. Then, 
 $$H_R(u^{\eps}_j)(x) = \textnormal{sup}_{{B}_\sizeofboundary(x)}u^{\eps}_j \geq u^{\eps}_j(\bar y)= f_j(\bar y) \geq \sigma.$$ Consequently, 
 from \eqref{Laplace_lowerbound_fi_lem}, we obtain \eqref{mainequations5.2} with $\bar\alpha  = 1$ and $\bar \beta = 0$.

 Next,   consider the case where $H_R$ is given by \eqref{H1}. 
  Let $\gamma>0$  depend on the modulus of continuity of $f_j$ so that $f_j > \frac{\sigma}{2}$ in $ B_{\sigma^\gamma}(\bar y) \cap \textnormal{supp }f_j$.
 Define $r_0 := \textnormal{min}\{\sigma^{\gamma}, \frac{r}{4}\}$. Then, for $z\in B_{r_0}(\bar y)$,   
  $$ |x-z|\leq |x-\bar y|+|\bar y-z|\leq \sizeofboundary-\frac{r}{2} +r_0 \leq \sizeofboundary-\frac{r}{4}.$$
In particular, $ B_{r_0}(\bar y)  \subset {B}_\sizeofboundary(x)$. 
Therefore,  using assumption \eqref{size_of_ball_intersect_support}, 
we get

\begin{eqnarray*}
\label{mainequations5.5}
H_R(u^{\eps}_j)(x) &=& \fint_{{B}_{\sizeofboundary}(x)}(u^{\eps}_j)^p(z)\, dz 
 \geq \frac{1}{|{B}_{\sizeofboundary}(x)|}\int_{{B}_{r_0}(\bar y) \cap \textnormal{supp }f_j}(u^{\eps}_j)^p(z) \,dz  \nonumber \\
 &=& \frac{1}{|{B}_{\sizeofboundary}(x)|}\int_{{B}_{r_0}(\bar y) \cap \textnormal{supp }f_j}f_j^p(z)\, dz  \nonumber 
 \geq \frac{1}{|{B}_{\sizeofboundary}(x)|}\int_{{B}_{r_0}(\bar y) \cap \textnormal{supp } f_j
 }\left(\frac{\sigma}{2}\right)^p dz \nonumber\\
 &\geq& C\sigma^p r_0^n\\
 &=&    C\min\Big\{\sigma^{p+\gamma}, \frac{\sigma^{p}r}{4}\Big\}\\
  &\geq& C\sigma^{p+\gamma}r.
\end{eqnarray*}
 Consequently, from \eqref{Laplace_lowerbound_fi_lem}, we obtain \eqref{mainequations5.2} with $\bar\alpha  = p+\gamma$ and $\bar \beta = 1$.
 This completes the proof of the lemma.
\end{proof}

\noindent The following result, a corollary of Lemma \ref{firstdecaylemma}, provides  interior Lipschitz estimates for the functions $u^{\eps}_i$,  which are uniform  in $\eps$.

\begin{corollary}
\label{Lipscitzestimate}      
      Let $(u^{\eps}_1,\ldots , u^{\eps}_K)$ be a  solution of  problem \eqref{main_problems}. 
      Let $\bar y$ be a point in $\Omega$ such that  $u^{\eps}_i(\bar y) = \sigma$, $d(\bar y, \textnormal{supp }f_j) \geq \sizeofboundary + \tau m\sizeofboundary r$ for  $j \neq i$,  and $d(\bar y, \partial \Omega) \geq 2\sizeofboundary r$, where $m = \frac{\sigma}{\textnormal{sup }_{\partial \Omega}f_i}$, $0 < r < 1$, and $\tau$, $\alpha$, and $\beta$, are given  in Lemma \ref{firstdecaylemma}. 
      Then, there exists $C > 0$ such that  for   $0<\eps \leq \sigma^{2\alpha}r^{2\beta}$, 
\begin{equation}
\label{mainequations5.6}
   |\nabla u^{\eps}_i| \leq \frac{C}{Rr}\quad  \text{ in } B_{\frac{\tau m\sizeofboundary r}{{8}}}(\bar y)
\end{equation}
and 
\begin{equation}
\label{mainequations5.7}
   \M^-(u^{\eps}_i) \rightarrow 0\quad  \text{ in } B_{\frac{\tau m\sizeofboundary r}{{2}}}(\bar y)
\end{equation}
      uniformly as $\eps \rightarrow 0^+$.
\end{corollary}

\begin{proof}
First,  note that ${B}_{\frac{\tau m R r}{2}}(\bar y) \subset  {B}_{2Rr}(\bar y) \subset \Omega,$ where the first inclusion follows from the fact that $\tau< 1$ and $m = \frac{\sigma}{sup_{\partial \Omega}f_i} \leq 1$. 

We claim that for any $z \in {B}_{\frac{\tau m \sizeofboundary r}{2}}(\bar y)$, we have 
\\\begin{equation}
\label{mainequations5.9}
   u^{\eps}_j(\bar x) \leq C e^{\frac{-c\sigma^{\alpha}r^{\beta}\sizeofboundary}{\eps}} 
\end{equation}
for all $\bar x \in B_\sizeofboundary(z)$ and all $j \neq i$.
 To see why  \eqref{mainequations5.9} holds,  assume $z \in {B}_{\frac{\tau m\sizeofboundary r}{2}}(\bar y)$ and 
 $\bar x \in {B}_\sizeofboundary(z)$. Then 
 $$|\bar x-\bar y| \leq \sizeofboundary+\frac{\tau m\sizeofboundary r}{2}.$$ Moreover, since $d(\bar y, \textnormal{supp }f_j) \geq \sizeofboundary+ \tau m\sizeofboundary r$, we have
  $$d(\bar x, \textnormal{supp }f_j) \geq \frac{\tau m\sizeofboundary r}{2}.$$ 
 Let $\Gamma^{\sigma, r}_i $ and $\Sigma^{\sigma, r}_{i,j}$ be defined as  in Lemma \ref{firstdecaylemma}. 
 Note that 
 $\bar y \in \Gamma^{\sigma, r}_i$  and $\bar x \in \Sigma^{\sigma, r}_{i,j}$. 
 Therefore,  by Lemma   \ref{firstdecaylemma}, 
estimate \eqref{mainequations5.9} follows. 

Now, using  \eqref{mainequations5.9}, for all $z\in B_{\frac{\tau m \sizeofboundary r}{2}}(\bar y) $ and for $0<\eps \leq \sigma^{2\alpha}r^{2\beta}$, we obtain
\begin{equation*}
\label{mainequations5.10}
 0 \leq \M^-(u^{\eps}_i(z)) \leq u^{\eps}_i(z) \frac{C e^{\frac{-c\sigma^{\alpha}r^{\beta}\sizeofboundary}{\eps}}}{\eps^2} \leq  \frac{C e^{-c\eps^{-\frac{1}{2}}R}}{\eps^2} \to 0   
\end{equation*}
as $\eps \rightarrow 0^+$, which proves  \eqref{mainequations5.7}.

It therefore remains to establish  \eqref{mainequations5.6}. 

 By  \eqref{mainequations5.7}, the function $u^{\eps}_i$ satisfies
$$ \M^-(u^{\eps}_i)=g\quad \text{in } B_{\frac{\tau m \sizeofboundary r}{2}}(\bar y),$$ 
with $\|g\|_{L^\infty( B_{\frac{\tau m \sizeofboundary r}{2}}(\bar y))}\leq C$ for some $C>0$ independent of $\eps$. Define the rescaled function
$$ v^{\eps}_i(x):=4 \frac{u^{\eps}_i\left(\frac{\tau m \sizeofboundary r}{4}x+\bar y\right)}{\tau mRr}.$$
Note that $v^{\eps}_i$ satisfies 
$$\M^-(v^{\eps}_i)=\overline g\quad \text{in } B_2(0),$$
with  $\|\overline g\|_{L^\infty( B_2(0))}\leq C$ and 
$$v^{\eps}_i(0)=\frac{4\sigma}{\tau mRr}=\frac{4\sup_{\partial\Omega}f_i}{\tau Rr}\leq \frac{C}{Rr}.$$
By the Harnack inequality,  Theorem \ref{Harnackthm}, 
$$ v^{\eps}_i(x)\leq C(n)(v^{\eps}_i(0)+C)\leq \frac{C}{Rr}\quad \text{for all }x\in B_1(0).$$
By Theorem \ref{W2pRegularity}  and Sobolev embeddings, we infer that  
\[
|\nabla  v^{\varepsilon}_i |\leq \frac{C}{Rr}\quad\text{in } B_{\frac12}(0).
\]
Estimate \eqref{mainequations5.6} follows. This completes the proof of the corollary.
\end{proof}

 Finally, the following result is a corollary of Lemma \ref{seconddecaylemma}.
\begin{corollary}
\label{first_part_of_maintheorem4}
    Let $(u^{\eps}_1,\ldots, u^{\eps}_K)$ be a solution of  problem \eqref{main_problems}. 
    For  $i = 1, \ldots, K$, define
    $$B_i:=\cup_{j\neq i}\{x \in \Omega\, : \, d(x, \textnormal{supp }f_j ) \leq  \sizeofboundary\}.$$
    Then   $u^{\eps}_i \rightarrow 0$ as $\eps \rightarrow 0^+$ in  $ B_i$. 
    
\end{corollary}
\begin{proof}
Let $x_0 \in B_i$. Then there exists $j \neq i$ such that $d(x_0, \textnormal{supp }f_j ) \leq \sizeofboundary $. First, assume that 
$d(x_0, \textnormal{supp }f_j ) < \sizeofboundary $. Observe that    
\[
\left\{x \in \Omega\, :\, d(x, \operatorname{supp} f_j) < \sizeofboundary \right\} \subset \bigcup_{r, \sigma > 0} \left\{x \in \Omega\, :\, d(x, \bar{\Gamma}_j^{\sigma}) \leq \sizeofboundary - r \right\},
\]
 where $\bar{\Gamma}_j^{\sigma}$ is defined as in  Lemma \ref{seconddecaylemma}. Hence, there exist $r, \sigma > 0$ such that 
\[
x_0 \in \left\{x \in \Omega : d(x, \bar{\Gamma}_j^{\sigma}) \leq \sizeofboundary - r \right\}.
\]
By Lemma \ref{seconddecaylemma}, it follows that 
\[
u^{\varepsilon}_i(x_0) \leq C e^{\frac{-c \sigma^{\alpha} r^{\beta}}{\varepsilon}}
\]
for some constants $C,c, \alpha, \beta > 0$. Letting $\varepsilon \to 0^+$ in the above inequality, we obtain 
\[
u^{\varepsilon}_i(x_0) \to 0 \quad \text{as } \varepsilon \to 0^+.
\]
Next,  assume that 
$d(x_0, \textnormal{supp }f_j ) = \sizeofboundary $. Consider the set  $A_r:=\{x \in \Omega : d(x, \operatorname{supp} f_j) \geq  \sizeofboundary-r \} $. Note that $A_r$  has an exterior tangent ball of radius $R-r$ at every point of
its boundary.
Moreover, we have just showed that $\sup_{\partial A_r\cap\Omega}u^\eps_i\to 0$ as $\eps\to0^+$.
 Since points $x$ in $\Omega$  such that $d(x, \operatorname{supp} f_j) =  \sizeofboundary$ are at distance $r$ from $\partial A_r$,  a barrier argument
 shows that there exist $ C,r_0>0$ such that  $u^\eps_i(x)\leq \sup_{\partial A_r\cap\Omega}u^\eps_i+ Cr$ for all $x\in B_{r_0}(x_0)\cap
 \{x \in \Omega\, :\, d(x, \operatorname{supp} f_j) =  \sizeofboundary \} $.
 Evaluating the latter inequality at $x=x_0$, sending first $\eps\to0^+$ and then $r\to0^+$, yields $u^\eps_i(x_0)\to0$ as $\varepsilon \to 0^+$.
 
This completes the proof of the corollary.
\end{proof}

\noindent {\bf Proof of Theorem \ref{main theorem3}:}
For each for $i=1\ldots, K$, define

 $$\Omega_i:=\{x \in \Omega\,: \,d(x, \textnormal{supp }f_j ) > \sizeofboundary \text{ for all }j \neq i \}.$$

\noindent{\bf Claim: }{\em  There exists a subsequence $\{u_i^{\eps_l}\}_l $ locally uniformly convergent in $\Omega_i $ as $l\to \infty$ to  a locally Lipschitz continuous function $u_i $. }

To prove the claim, fix  $r \in (0,\sizeofboundary)$ and define 
$$\Omega^r_i:=\{x \in \Omega\,: \,d(x, \partial \Omega ) > 2\sizeofboundary r,\, d(x, \textnormal{supp }f_j) \geq \sizeofboundary+\tau \sizeofboundary r\text{ for all }j \neq i\},$$
where $\tau$ is defined as in Lemma \ref{decaylemmaprelim}. 
 Note that $\Omega_i=\cup_{r\in(0,R)}\Omega^r_i$. 
  Let $\alpha$ and $\beta$ be   as in Lemma \ref{firstdecaylemma}. Fix  $\theta < \frac{1}{2\alpha}$, and define $\sigma_{\eps} := \eps^{\theta}$. Observe  that 
  $$\eps = \sigma_{\eps}^{\frac{1}{\theta}} = \sigma_{\eps}^{2\alpha}\sigma_{\eps}^{\frac{1}{\theta}-2\alpha} = \sigma_{\eps}^{2\alpha}\eps^{\theta(\frac{1}{\theta}-2\alpha)}$$ and note that $\frac{1}{\theta} -2\alpha > 0$.
   Therefore, we can choose $\eps_0 =\eps_0(r)$  such that $\eps \leq \sigma_{\eps}^{2\alpha}r^{2\beta}$ whenever $0<\eps < \eps_0$. 
   Now define
   $$v^{\eps}_i := (u^{\eps}_i-\sigma_{\eps})_+.$$
   Then,  the functions $v^\eps_i$ are uniformly  Lipschitz continuous on $\Omega^r_i$. Indeed, 
    if $u^{\eps}_i(x) \leq \sigma_\eps$, then  clearly $v^{\eps}_i = 0$. 
    Next, let   $x\in  \Omega^r_i$ be  such that $u^{\eps}_i(x) > \sigma_\eps = \eps^{\theta}$.   Define $\sigma:=u^{\eps}_i(x)$ and  $m := \frac{\sigma}{\sup_{\partial \Omega} f_i} \leq 1$.
  Then,    for all  $j\neq i$, $d(x, \supp f_j) \geq \sizeofboundary+\tau\sizeofboundary r \geq \sizeofboundary+ \tau m\sizeofboundary  r$.  In addition, $d(x, \partial \Omega) > 2\sizeofboundary r$ and 
  for $0<\eps<\eps_0$ we have $\eps \leq \sigma_{\eps}^{2\alpha}r^{2\beta}  \leq \sigma^{2\alpha}r^{2\beta}$. Thus, we can apply Corollary \ref{Lipscitzestimate} to obtain 
 $$|\nabla v^{\eps}_i(x)|=|\nabla u^{\eps}_i(x)|\leq \frac{C}{rR}.$$

 Therefore, by the Arzel{\`a}-Ascoli Theorem we may extract a subsequence  $\{v_i^{\eps_l}\}_l $ uniformly convergent to a Lipschitz
continuous function $u_i$ in $\Omega_i^r$ as $l\to\infty$.  Since $|u^{\eps}_i-v^{\eps}_i|\leq \eps^\theta$,  the same  subsequence  $\{u_i^{\eps_l}\}_l $ converges uniformly to $u_i$ in $\Omega_i^r$.
Taking a sequence $r_k\to0^+$ as $k\to\infty$
and using a diagonal  argument, we can find a subsequence of   $\{u^{\eps}_i\}_{0 < \eps < \eps_0}$  converging locally uniformly to a locally Lipschitz function $u_i$ in $\Omega_i$. This concludes the proof of  the claim.  
  

 Now let $B_i:=\Omega\setminus  \Omega_i$. Then,   by Corollary \ref{first_part_of_maintheorem4} $u_i^{\eps_l}\to 0$ in $B_i$ as $l\to \infty$. 
 
 Combining this with the previous claim proves convergence of the subsequence 
$\{u_i^{\eps_l}\}_l $ to a function $u_i$ which is locally Lipschitz   in $\Omega_i$ and in the interior of $B_i$.  
 
 To conclude the proof of part (1)  of the theorem, it remains to show that the function $u_i$  is locally Lipschitz on $\partial \Omega_i\cap \Omega=\partial B_i\cap \Omega$.

 First, note that the second part of the theorem  follows immediately from the proof of the claim and Corollary \ref{Lipscitzestimate}.

 Now, let $x_0\in \partial \Omega_i\cap \Omega$, then $u_i(x_0)=0$. 
If $x_0\not\in\partial\{u_i>0\}$ then in a neighborhood
of $x_0$ $u_i=0$  and of course $u_i$ is Lipschitz in that neighborhood. Suppose instead that $x_0\in\partial\{u_i>0\}$.  By the definition of $\Omega_i$,  there exists an exterior ball  of radius $R$ at every point of 
$\partial\Omega_i\cap\Omega$, and we have $\mathcal{M}^-(u_i)=0$ on $\{u_i>0\}$. Then,  a standard  barrier argument 
shows that that there exists $r_0,\,C>0$ such that 
$0\leq u_i(x)=u_i(x)-u_i(x_0)\leq C|x-x_0|$ for all $x\in B_{r_0}(x_0)$. This establishes the local Lipschitz continuity of $u_i$ near $x_0$, completing the proof of part (1) of the theorem.

We now prove part (3) of Theorem \ref{main theorem3}. Let $x_0\in  \Omega \cup (\partial \Omega)_{\leq \sizeofboundary}$ be such that $u_i(x_0)>0$. Let us show that if $y_0\in\Omega$ is such that 
$|x_0-y_0|\leq R$ then $u_j(x_0)=0$ for all $j\neq i$. 

If $x_0\in  (\partial \Omega)_{\leq \sizeofboundary}$ then this  follows from Corollary \ref{first_part_of_maintheorem4}.

Assume now that  $x_0\in \Omega$. Let $0<r<1$ be such that $d(y_0,\partial\Omega)\geq 2Rr$. Let $\sigma_l:=u_i^{\eps_l}(x_0)$, then $\sigma_l\geq u_i(x_0)/2>0$ for $l$ sufficiently large. 
By Lemma \ref{firstdecaylemma},    there exist $C,c,\alpha,\beta>0$ such that $u_j^{\eps_l}(y_0)\leq C e^{\frac{-c\sigma_l^{\alpha}r^{\beta}R}{\eps}}$ for all $j\neq i$. 
Letting $l$ go to infinity we obtain  $u_j(y_0)=0$. 

This completes the proof of (3) and of the theorem.  $\Box$

\section{A Semiconvexity Property of the Free Boundary}

We have shown in the previous section that there is a subsequential limit $(u_1, \ldots, u_K)$ of $\{(u^{\eps}_1, \ldots, u^{\eps}_K)\}_{\eps>0}$. In this section, we study the geometry of the sets
$$S(u_i):=\{x\in\Omega\,:\, u_i > 0\},\quad i=1,\ldots, K,$$ and the corresponding free boundaries $\partial S(u_i)\cap\Omega$.  We will show that each set  $S(u_i)$ has finite perimeter (Theorem \ref{fp}), and that its free boundary
satisfies a semi-convexity property  (Theorem \ref{main theorem5}). 

  The proof of the following result is based on the strong minimum principle, Theorem \ref{strong_maximum_principle}.

\begin{lemma}
\label{main_containment_lemma}
Let $F(u_i) := \{x \in \R^n: d(x, S(u_i)) \geq \sizeofboundary\}$, and set $N(u_i) := \{x \in \Omega: d(x, F(u_i)) > \sizeofboundary\}$. Then   $\partial S(u_i)\subset \partial N(u_i)$. 
\end{lemma}
\begin{proof}

To show that $\partial  S(u_i) \subset \partial N(u_i)$, it suffices to prove 
that every connected component of  $S(u_i)$ is also a connected component of $N(u_i)$. 

First, observe that $S(u_i)\subset N(u_i)$. Indeed  if $p\in S(u_i)$, and since $S(u_i)$ is open, we have that $d(p,F(u_i))>R$. Thus, $p\in N(u_i)$. 

Now, for all $\sigma > 0$,  consider the sets
$$S_\sigma(u_i):=\{x\in \Omega\,:\,u_i>\sigma\},$$
 $$F_{\sigma}(u_i) := \{x \in \R^n\,:\, d(x, S_\sigma(u_i))  \geq \sizeofboundary\}$$
  and $$N_{\sigma}(u_i) := \{x \in \Omega\,:\, d(x, F_{\sigma}(u_i)) > \sizeofboundary\}.$$ 
 We claim  that
\begin{equation}
\label{equality_of_union_of_balls}
    (F_{\sigma}(u_i))^c = 
    \bigcup_{x\in S_\sigma(u_i)}{B}_{\sizeofboundary}(x)  =   \bigcup_{x \in N_{\sigma}(u_i)}{B}_{\sizeofboundary}(x).
\end{equation}

To check the first equality of \eqref{equality_of_union_of_balls}, note that if $p \in \cup_{x \in S_\sigma(u_i)}{B}_{\sizeofboundary}(x)$, then $p \in {B}_{\sizeofboundary}(x)$ for some 
$x \in S_\sigma(u_i)$. It follows that $d(p,S_\sigma(u_i)) < \sizeofboundary$, so $p\in  (F_{\sigma}(u_i))^c$. Conversely, if $p \in(F_{\sigma}(u_i))^c$, then $d(p, S_\sigma(u_i)) < \sizeofboundary$, so $p \in {B}_{\sizeofboundary}(x)$ for some $x \in S_\sigma(u_i)$. 
This proves that $(F_{\sigma}(u_i))^c = \cup_{x \in S_\sigma(u_i)}{B}_{\sizeofboundary}(x)$. 
To prove the second equality, note that if $p \in \cup_{x \in N_{\sigma}(u_i)}{B}_{\sizeofboundary}(x)$, then $p \in {B}_{\sizeofboundary}(x)$ for some $x \in N_{\sigma}(u_i)$. Since
$|x-p|<R$, by the definition of $ N_{\sigma}(u_i)$ we must have that  $p\in (F_{\sigma}(u_i))^c$. 
 On the other hand,  since  $S_\sigma(u_i) \subset N_{\sigma}(u_i)$, it follows that $ \cup_{x \in S_\sigma(u_i)}{B}_{\sizeofboundary}(x) \subset \cup_{x \in N_{\sigma}(u_i)}{B}_{\sizeofboundary}(x) $. 
 We have shown  that $\cup_{x \in N_{\sigma}(u_i)}{B}_{\sizeofboundary}(x) =\cup_{x \in S_{\sigma}(u_i)}{B}_{\sizeofboundary}(x)$. 
 Consequently, 
\eqref{equality_of_union_of_balls} holds.

    Next, we claim that for all $\sigma > 0$,
  \begin{equation}
  \label{u_i_solves_the_Mminus}
  \M^-(u_i) = 0 \text{ in } N_{\sigma}(u_i)
  \end{equation}
provided $N_{\sigma}(u_i)$ is nonempty. 
Let $\{u_i^{\eps_l} \}_l$ be a subsequence converging locally uniformly  to $u_i$  in $\Omega$. Fix $x \in S_\sigma(u_i)$, then $\sigma_l:=u_i^{\eps_l}(x)\geq\sigma$ for $l$ sufficiently  large. Moreover, 
by Theorem \ref{main theorem3}-(3), we know that $d(x,\supp f_j) > \sizeofboundary$, for all $j \neq i$. Therefore, there exists $r>0$ such that 
 $x \in \Gamma^{\sigma_l,r}_i$ and $B_R(x)\subset  \Sigma^{\sigma_l,r}_{i,j}$  where $\Gamma^{\sigma_l,r}_i$ and $\Sigma^{\sigma_l,r}_{i,j}$ are defined as in Lemma \ref{firstdecaylemma}. By Lemma \ref{firstdecaylemma}, it follows that 
 $$u_j^{\eps_l} \leq Ce^{\frac{-c\sigma_l^\alpha r^\beta \sizeofboundary}{\eps_l}}\leq  Ce^{\frac{-c\sigma^\alpha r^\beta \sizeofboundary}{\eps_l}}\quad\text{ in }{B}_{\sizeofboundary}(x).$$
  Thus,  for all  $j\neq i$, 
 $$u_j^{\eps_l} \leq Ce^{\frac{-c\sigma^\alpha r^\beta \sizeofboundary}{\eps_l}}\quad\text{ in } \bigcup_{x \in S_{\sigma}(u_i)}{B}_{\sizeofboundary}(x).$$
Recalling \eqref{equality_of_union_of_balls}, we conclude that  for all $j\neq i$, 
$$\frac{H_R(u^{\eps_l}_j)}{\eps_l^2}\to 0\quad \text{in } N_{\sigma}(u_i),$$
which implies the desired result  \eqref{u_i_solves_the_Mminus}.

      Finally, we prove that every connected component of $S(u_i)$ is also a connected component of $N(u_i)$. To see this,  let $A$ be a connected component of  $S(u_i)$.
      Since $ S(u_i)\subset N(u_i),$ there exists $B$  connected component  of $N(u_i)$ such that $A \subset B.$ 
      
      For $\sigma > 0,$ 
      define 
      $$A_{\sigma} := A\cap S_\sigma(u_i),\quad
      B_{\sigma} := B \cap  N_\sigma(u_i).$$  By \eqref{u_i_solves_the_Mminus} $\M^-(u_i) = 0$ in $B_{\sigma}$. Moreover, since $A_{\sigma} \subset B_{\sigma}$, we know  that  $u_i\not \equiv 0$ in $B_\sigma$. 
      Then, by the  strong minimum principle, Theorem \ref{strong_maximum_principle}, it follows  that $u_i>0$ in $B_{\sigma}$, that is $B_{\sigma}\subset S(u_i)$.
     Since this holds  for every $\sigma>0$,   we conclude that $ B\subset S(u_i)$. But we already had that $A\subset B$, and $A$  was a connected component of $S(u_i)$. 
Since $B$ is also connected, we must have $A=B$. 
     
      This completes   the proof of the lemma.      \end{proof}

{\bf Proof of Theorem \ref{main theorem5}:}
For $i=1,\ldots, K$, let $S(u_i)$ and $N(u_i)$ be defined as in Lemma  \ref{main_containment_lemma}. By definition of $N(u_i)$, for every $x\in \partial N(u_i)\cap\Omega$ there exists an exterior  ball of radius $R$ tangent at $x$. Theorem \ref{main theorem5} then immediately follows from Lemma  \ref{main_containment_lemma}.
 $\Box$

The following theorem is shown in the Appendix.
\begin{theorem}
\label{key}
Let $E $ be a compact subset of $\mathbb{R}^n$, and let $E_t := \{x \in \mathbb{R}^n: d(x,E) < t\}$, $t>0$. Then $E_t$ has finite perimeter. 
\end{theorem}

{\bf Proof of Theorem \ref{fp}:}
Let $F(u_i)$ and $N(u_i)$ be defined as in Lemma \ref{main_containment_lemma}.  Note that 
$$ N(u_i)= \{x\in \R^n\,:\,d(x,\partial F(u_i)) > R\}\cap\Omega$$
and $\partial F(u_i)$ is a compact set. By Theorem \ref{key},   $N(u_i)$ has finite perimeter. Theorem \ref{fp} then immediately follows from Lemma \ref{main_containment_lemma}.
$\Box$.

\section{Appendix: a proof of Theorem \ref{key}}
Theorem \ref{key} can be proven either using PDE techniques, as in \cite{CL2}, or using techniques 
 from geometric measure theory, as in \cite{Daniel_Kraft}. Here we explain the proof of Theorem \ref{key} given in 
\cite{Daniel_Kraft}, which is based on a covering argument. 

For a set $E \subset \mathbb{R}^n$,  we denote by $d_{E}$  the distance function from $E$, given by
\begin{equation*}
d_{E}(x) := \inf_{y \in E}|x-y|\quad  \text{ for all } x \in \mathbb{R}^n.
\end{equation*}
 We also denote by $P(E)$ the perimeter of $E$, defined as 
 the $\mathcal{H}^{n-1}$-measure of the reduced boundary of $E$. For the definition of reduced boundary we refer the reader to \cite{E}.
 
For $t> 0$,  we define
\begin{equation}\label{E_tdef}
E_t := \{x \in \mathbb{R}^n\,: \,
 d_{E}(x) < t\}, 
\end{equation} 
\noindent  and
 \begin{equation}\label{U_tdef}
     U_t := \{x \in \mathbb{R}^n: 0< d_{E}(x) < t\} = \bigcup_{x_0 \in \partial E } B_t(x_0)\setminus (\partial E \cup E).
 \end{equation}
 Note that $E_t$ is open and $\partial E_t = \{x \in \mathbb{R}^n: 
 d_{E}(x) = t\} = d_{E}^{-1}(\{t\})$.

   We also recall, that 
   $$ P(E_t) \leq \mathcal{H}^{n-1}(\partial E_t) =\lim_{\delta\to0^+} \mathcal{H}^{n-1}_\delta(\partial E_t),$$ 
   where 
   $$ \mathcal{H}^{n-1}_\delta(\partial E_t):=\inf\left\{\sum_{i=1}^\infty \omega_{n-1}r_i^{n-1}\,:\, \partial E_t\subset \bigcup_{i=1}^\infty A_i, \,r_i=\frac12\sup_{x,y\in A_i}|x-y|\leq\delta\right\}$$
  and  $\omega_{n-1}$ denotes the volume of the unit ball in $\R^{n-1}$, see for instance \cite{E}.
  

 Let $x_0, x \in \mathbb{R}^n$ and $\phi \in [0,\frac{\pi}{2}].$ We define the  open "sector" with center at $x_0$ and radius $t:= |x- x_0|$ by
 \begin{equation*}
 A_{\phi}(x_0,x) := \{y \in \mathbb{R}^n: 0 < |x_0-y| < |x_0-x|, (y-x_0)\cdot (x-x_0) > |x_0-x||x_0-y|\textnormal{cos}\phi\}.
 \end{equation*}
 For all $0 \leq r_1 \leq r_2 \leq 1,$ we further define 
 \begin{equation*}
 A_{\phi}(x_0,x; r_1,r_2) := \{y \in A_{\phi}(x_0,x): r_1|x_0-x| < |x_0-y| < r_2 |x_0-x|\}
 \end{equation*}

We will need the following properties of sectors proven in \cite{Daniel_Kraft}.

\begin{enumerate}
 \item For a fixed $t > 0,$ there exists $0<\delta_0 <t$ and a dimensional constant $C > 0$ such that
\begin{eqnarray}
    \delta^{n-1}\omega_{n-1} \leq C\frac{|A_{\phi(\delta)}(x_0,x)|}{t} 
\end{eqnarray}
for all $\delta \in (0,\delta_0)$ and arbitrary $x_0, x \in \mathbb{R}^n$ with $|x_0-x| = t.$ Here $\phi(\delta) := \textnormal{arccos}(1-\frac{\delta^2}{2t^2}).$

\item Let $x_0, x \in \mathbb{R}^n, \phi \in [0,\frac{\pi}{2}]$ and $z \in \overline {A_{\phi}(x_0,x)}.$ Denote by $t := |x_0-x|$ and $|x_0-z| = rt$ with some $r \in [0,1].$ Then 
$t(1-r) \leq |x-z| \leq t(1-r+\phi r) \leq t(1-r+\phi)$.

\item Fix a $t > 0.$ Then there exist $0<\delta_0 <t$, $0< a <1 $ and $0 < r_1 < r_2 <1$ such that 
\begin{eqnarray}
    \overline{A_{a\phi(\delta)}(x_0,x; r_1,r_2)} \cap \overline{ A_{a\phi(\delta)}(y_0,y;r_1,r_2)}=\emptyset
\end{eqnarray}
for all $\delta \in (0,\delta_0)$ and $x_0,y_0,x,y \in \mathbb{R}^n$ with $t = |x-x_0| = |y-y_0|,$ whenever $\overline{B_{\delta}(x)} \cap \overline{ B_{\delta}(y)}= \emptyset$ and $t \leq \min(|x_0-y|,|y_0-x|).$ The constants $a, r_1, r_2$ are explicit and do not depend on any other values, and here $\phi(\delta) := \textnormal{arccos}(1-\frac{\delta^2}{2t^2}).$

\end{enumerate}

 Theorem \ref{key} is a consequence of the following result.

\begin{theorem}
\label{key2}
(\cite{Daniel_Kraft}, Theorem 2)
Let $E \subset \mathbb{R}^n$ be a compact set, and let $E_t$ and $U_t$ be defined as is \eqref{E_tdef} and \eqref{U_tdef}, respectively. Then there exists a dimensional constant $C > 0$ such that 
\begin{eqnarray}
    P(E_t) \leq \mathcal{H}^{n-1}(\partial E_t) \leq C\frac{|U_t|}{t}
\end{eqnarray}
for all $t > 0.$
\end{theorem}
\begin{proof}
We prove only  the second inequality since the first inequality is well known. Fix  $t > 0.$ 
 Let $\delta_0$ be as in property (3) above and
assume  $\delta \in (0,\delta_0).$ 

Clearly, $\partial E_t \subset \bigcup_{x \in \partial E_t} \overline{ B_{\delta}(x)}.$ By  Vitali's Covering Theorem (see for instance \cite[Theorem 1.24]{E}), there exists a countable subset $J \subset \partial E_t$ such that 
\begin{eqnarray}
\label{containment}
    \partial E_t \subset \bigcup_{x \in J} \overline{B_{5\delta}(x)},
\end{eqnarray}
and the family $\{B_{\delta}(x)\}_{x \in J}$ is pairwise disjoint.

Since $E$ is compact, for each $x\in J$ there exists $x_0\in \partial E$ such that $|x-x_0|=t$. Let us define  
$A_x := \overline{A_{a\phi(\delta)}(x_0,x;r_1,r_2)}$,  with $a$, $r_1$ and $r_2$ as in property (3). Then, for all $z\in A_x$, we have 
\begin{eqnarray}
\label{A_xdistancefromabove}
d_E(z)\leq |z-x_0|\leq tr_2<t.
\end{eqnarray}
Moreover, since the family $\{B_{\delta}(x)\}_{x \in J}$ is pairwise disjoint, by property (3)  we have that $A_x$ and $A_y$ are disjoint for all $x, y \in J$, $x \neq y$. 

We now show that, 
\begin{eqnarray}
\label{Sx_are_in_the_newly_created_volume}
    A_x \cap E = \emptyset
\end{eqnarray}
provided $\delta > 0$ is small enough.  
 To see this,
pick any $z \in A_x$, then  $|x_0-z| = rt$ for some $r \in [r_1,r_2].$ By property (2), we have that 
$$|x-z| \leq t(1-r+a\phi(\delta)) < t(1-r+\phi(\delta))\leq t(1-r_1+\phi(\delta))\leq t,$$ 
if $\delta$ is so small that $\phi(\delta) \leq r_1.$ Since $B_t(x)\subset E^c$,  this implies that $z \notin E.$

By \eqref{A_xdistancefromabove} and \eqref{Sx_are_in_the_newly_created_volume}  it follows that each $A_x$ is contained in $U_t$ and 

\begin{eqnarray}
\label{estimate_on_the_volume_of_S_phi}
    \sum_{x \in J} |(A_{\phi(\delta)}(x_0,x))| \leq C^{\prime} \sum_{x \in J} |A_x| = C^{\prime} \big|\bigcup_{x \in  J} A_x\big | \leq C^{\prime} |U_t|,
\end{eqnarray}
for some suitable constant $C^{\prime}>0,$ as long as $\delta > 0$ is small enough.
Now combine \eqref{containment},    \eqref{estimate_on_the_volume_of_S_phi} and  property (1)  above to obtain
\begin{eqnarray*}
\label{estimate_on_Hausdorff_measure_on_gamma_t}
    \mathcal{H}^{n-1}_{5\delta} (\partial E_t) \leq \sum_{x \in J} (5\delta)^{n-1} \omega_{n-1} \leq 5^{n-1} \frac{C^{\prime \prime}}{t} \sum_{x \in J} |A_{\phi(\delta)}(x_0,x)|
    \leq 5^{n-1}C^{\prime}C^{\prime\prime}\frac{|U_t|}{t}, 
\end{eqnarray*}
 for some  $C^{\prime\prime}>0$. 
Letting $\delta \rightarrow 0^+$,  we obtain the desired result.  \end{proof}

\section*{Acknowledgment}

The second author was supported by NSF  Grant DMS-2155156, “Nonlinear PDE Methods in the Study of Interphases,” and by a grant from the Simons Foundation (MPS-TSM-00697812).

\end{document}